\documentclass[12pt]{amsart}

\usepackage[backref=page]{hyperref}

\usepackage{amssymb, hyperref,xcolor, mathrsfs, cancel}
\usepackage[capitalise]{cleveref}

\usepackage[margin=1in]{geometry}

\newtheorem{theorem}{Theorem}[section]

\theoremstyle{remark}
\newtheorem{remark}[theorem]{Remark}

\numberwithin{equation}{section}

\newcommand{\mo}{{-1}}

\newcommand{\calG}{\ensuremath{\mathcal{G}}}

\newcommand{\calN}{\ensuremath{\mathcal{N}}}
\newcommand{\calO}{\ensuremath{\mathcal{O}}}

\begin{document}

\title{A spectral bound for vertex-transitive graphs}
 
\author{Arindam Biswas}
\address{Department of Mathematical Sciences,
	University of Copenhagen,
	Universitetsparken 5,
	DK-2100 Copenhagen, Denmark}
\email{ab@math.ku.dk}
\thanks{}

\author{Jyoti Prakash Saha}
\address{Department of Mathematics, Indian Institute of Science Education and Research Bhopal, Bhopal Bypass Road, Bhauri, Bhopal 462066, Madhya Pradesh,
India}
\curraddr{}
\email{jpsaha@iiserb.ac.in}
\thanks{}



\begin{abstract}
For any finite, undirected, non-bipartite, vertex-transitive graph, we establish an explicit lower bound for the smallest eigenvalue of its normalised adjacency operator, which depends on the graph only through its degree and its isoperimetric constant. 
\end{abstract}

\maketitle

\section{Introduction}

The graphs that are considered in this article are undirected, finite, connected and regular. The properties of the spectra of a graph, viz., the distribution of eigenvalues of its adjacency operator or its Laplacian operator, encode a lot of information about its structural properties. Thus, they are an important object of study not only in mathematics, but in other sciences as well where graph theory is applied. Some of the structural properties, e.g., expansion, diameter, Hamiltonicity etc. (to name a few), depend on the spectral gap, which depending on the context can mean either the difference between the largest and the second largest eigenvalues of its normalised adjacency operator, or the difference between the largest and the second largest eigenvalues (in absolute value) of its normalised adjacency operator. The study of spectral bounds, i.e., bounds for the spectral gap in terms of various invariants of a graph, is an important direction of research. 

From the discrete Cheeger--Buser inequality, established by Alon and Milman \cite{AlonMilmanIsoperiIneqSupConcen}, it follows that the second largest eigenvalue of the normalised adjacency operator of a finite graph is bounded uniformly away from $1$ in terms of its degree and its isoperimetric constant. 

In the case of Cayley graphs, a uniform lower bound for the smallest nontrivial eigenvalue was established qualitatively by Breuillard--Green--Guralnick--Tao in \cite[Appendix E]{BGGTExpansionSimpleLie}. Later, a quantitative version of this fact was established by the first author \cite{BiswasCheegerCayley} and the same phenomenon with explicit bounds were also shown to hold true for variants of Cayley graphs, in prior works of the authors \cite{CheegerCayleySum, CheegerTwisted}. Recently, an improved bound for Cayley graphs has been given by Moorman--Ralli--Tetali \cite{CayleyBottomBipartite}. However, it is known that the above fact doesn't hold true for general regular graphs. This leads to the question that for which other subclass of regular graphs does the above phenomenon occur. In particular, the situation for vertex-transitive graphs was unknown. In this article, we address this issue.

\begin{theorem}
\label{Thm:Intro}
For any finite, undirected, non-bipartite, vertex-transitive graph of degree $d$ having isoperimetric constant $h$, the nontrivial spectrum of its normalised adjacency operator is contained in $\left(-1 + \frac{h^4}{2^9 d^{10}}, 1- \frac{h^2}{2d^2}\right]$. 
\end{theorem}

The above result proves that combinatorial expansion implies two-sided spectral expansion for vertex-transitive graphs. Such a result was first proved by Breuillard--Green--Guralnick--Tao for Cayley graphs \cite[Appendix E]{BGGTExpansionSimpleLie}. 

For undirected, non-bipartite, vertex-transitive graphs, Theorem \ref{Thm:Intro} gives a lower bound for the smallest eigenvalue of the normalised adjacency operator in terms of the isoperimetric constant and the degree, thus, providing an analogue of the discrete Buser inequality for the smallest eigenvalue of the normalised adjacency operator. It would be interesting to investigate whether an analogue of the discrete Cheeger inequality holds for the smallest eigenvalue of the normalised adjacency operator, i.e., whether an upper bound for the smallest eigenvalue of the normalised adjacency operator in terms of the isoperimetric constant and the degree can be established for undirected, non-bipartite, vertex-transitive graphs, or even for the special case of Cayley graphs.  

\section{Proof of the spectral bound}

Theorem \ref{thmPrincipalk1} states that if the vertex set of an undirected, regular, non-bipartite graph $(V, E)$ carries a transitive action of a group $\calG$, and there is a nice interplay between the action of $\calG$ and the neighbourhoods in the square graph of $(V, E)$ (cf. condition (4), Theorem \ref{thmPrincipalk1}), and $\calG$ acts through ``almost graph automorphisms'' (i.e., the neighbourhood of the image of any given vertex $v$ under an element $g$ of $\calG$, need not necessarily be the image of the neighbourhood of $v$ under $g$, but it consists of the images of the neighbours of $v$ under certain homomorphic images of $g$ (cf. condition (3), Theorem \ref{thmPrincipalk1} for a precise statement)), then the smallest eigenvalue of the normalised adjacency operator of $(V, E)$ is bounded away from $-1$ in terms of its degree and its isoperimetric constant. 

The novelty in our approach is the crucial observation that for a vertex-transitive graph $(V, E)$, although the group $\mathrm{Aut}(V, E)$ of its automorphisms need not act on $V$ through ``almost graph automorphisms'', but the fact that $\mathrm{Aut}(V, E)$ preserves the adjacency relations, implies that a positive proportion (depending on the degree of $(V, E)$) of elements of $\mathrm{Aut}(V, E)$ acts on a given vertex through a variant (as in Equation \eqref{Eqn:QuasiAutoIJ}) of the ``almost graph automorphism'' condition (as in condition (3), Theorem \ref{thmPrincipalk1}), and that the validity of Equation \eqref{Eqn:QuasiAutoIJ} for enough elements of $\mathrm{Aut}(V, E)$ yields a lower bound on the spectra in the desired form. In the course of the proof, Equation \eqref{Eqn:QuasiAutoIJ} allows us to deal with vertex-transitive graphs in general, without imposing the supplementary assumption that their automorphism groups act through ``almost graph automorphisms''. 

In the following, $(V, E)$ denotes a finite graph (which may contain multiple edges, and even multiple loops at certain vertices). The neighbourhood of a subset $V'$ of $V$ in $(V, E)$ is denoted by $\calN(V')$. We assume that $(V, E)$ is undirected and regular of degree $d$, i.e., its adjacency matrix is symmetric and has the constant vector $[1, \ldots, 1]$ as an eigenvector with eigenvalue $d$. 

\begin{theorem}[{\cite[Theorem 3.2]{CheegerTwisted}}]\label{thmPrincipalk1}
Suppose there exist permutations $\theta_1, \ldots, \theta_d: V\to V$ such that the vertices $v, \theta_i(v)$ are adjacent in $(V, E)$ for any $v\in V$ and $1\leq i \leq d$, and that $\calN(v)$ is equal to $\cup_{i=1}^d \{\theta_i(v)\}$ for any $v\in V$. Assume that $V$ carries a left action of a group $\calG$ such that the following conditions hold. 
\begin{enumerate}
\item No index two subgroup of $\calG$ acts transitively on $V$.

\item The action of $\calG$ on the set $V$ is ``transitive of order $t$'' in the sense that for each $(u, v)\in V\times V$, the equation $gu = v$ has exactly $t$ distinct solutions for $g\in \calG$. 

\item For each $\theta_i$ with $1\leq i\leq d$ and each $v\in V$, there is an automorphism or an anti-automorphism $\psi_{i,v}$ of the group $\calG$ such that one of 
$$\theta_i(g\cdot v) =  \psi_{i,v}(g) \cdot \theta_i(v)$$
and 
$$\theta_i(g\cdot v) =  \psi_{i, v}(g^\mo) \cdot \theta_i(v)$$
holds for any $g\in \calG$. 

\item For any $\tau\in \calG$, the set 
$
\calN (\calN(\tau(A)))$ is contained in $\tau(\calN(\calN(A)))
$.
\end{enumerate}
Assume that $|V|\geq 4$, the graph $(V, E)$ is non-bipartite and it is an $\varepsilon$-vertex expander for some $\varepsilon>0$. Then the nontrivial eigenvalues of the normalised adjacency operator of $(V, E)$ are greater than $-1 + \ell_{\varepsilon, d}$ with 
$$\ell_{\varepsilon, d}
=\frac{\varepsilon^4}{2^{12} d^8}.$$
\end{theorem}

\begin{remark}
\label{Remark}
Note that for any automorphism $g$ of the graph $(V, E)$, it follows that for any vertex $v\in V$, the neighbourhood of $g(v)$ is equal to the image of the neighbourhood of $v$ under $g$, i.e., $\calN(g (v)) = g(\calN(v))$, and in particular, condition (4) of Theorem \ref{thmPrincipalk1} holds. Furthermore, if the group $\calG$ (as in Theorem \ref{thmPrincipalk1}) acts transitively on $V$, then the stabilisers of the elements of $V$ are of the same size, and hence a transitive action of $\calG$ of $V$ is transitive of order $t$ for some integer $t$. 
\end{remark}

We would like to prove that when the group $\calG$, as in Theorem \ref{thmPrincipalk1}, consists of graph automorphisms of $(V, E)$, then the conclusion of Theorem \ref{thmPrincipalk1} holds only under condition (1) and the assumption that $\calG$ acts transitively on $V$. More precisely, we establish the following result. 

\begin{theorem}\label{Thm:theta}
Assume that there exist permutations $\theta_1, \ldots, \theta_d: V\to V$ such that the vertices $v, \theta_i(v)$ are adjacent in $(V, E)$ for any $v\in V$ and $1\leq i \leq d$, and that $\calN(v)$ is equal to $\cup_{i=1}^d \{\theta_i(v)\}$ for any $v\in V$. Let $\calG$ be a subgroup of the group of automorphisms of the graph $(V, E)$ such that the following conditions hold. 
\begin{enumerate}
\item No index two subgroup of $\calG$ acts transitively on $V$.
\item The group $\calG$ acts transitively on $V$. 
\end{enumerate}
Assume that $|V|\geq 4$, the graph $(V, E)$ is non-bipartite and it is an $\varepsilon$-vertex expander for some $\varepsilon>0$. Then the nontrivial eigenvalues of the normalised adjacency operator of $(V, E)$ are greater than $-1 + \ell_{\varepsilon, d}$ with 
$$\ell_{\varepsilon, d}
= 
\frac{\varepsilon^4}{2^{9} d^{10}}.$$
\end{theorem}

As explained in Remark \ref{Remark}, under the hypothesis of Theorem \ref{Thm:theta}, conditions (1), (2), (4) of Theorem \ref{thmPrincipalk1} hold. If we could show that the steps of the proof of Theorem \ref{thmPrincipalk1}, that use condition (3), also hold under the hypothesis of Theorem \ref{Thm:theta}, then Theorem \ref{Thm:theta} would follow. We prove that this can indeed be achieved. 

\begin{remark}
\label{Remark:Summary}
Before proceeding to the proof of Theorem \ref{Thm:theta}, let us provide a summary of the proof of Theorem \ref{thmPrincipalk1}, highlighting the intermediate steps where condition (3) has been used. The proof of Theorem \ref{thmPrincipalk1} was obtained by the method of contradiction. First, it is assumed that the normalised adjacency operator of $(V, E)$ admits an eigenvalue in $[-1, -1 + \ell_{\varepsilon, d}]$. We considered a subset $H = H_{\varepsilon, d}$ of $\calG$ and using condition (4), we proved that $H$ is a subgroup of $\calG$ of index two. Using condition (1) of Theorem \ref{thmPrincipalk1}, it follows that the action of $H$ on $V$ has two orbits, and an orbit $\calO$ among them has certain specific properties, which yields that for any $1\leq i \leq d$, the inequality 
$$|\theta_i(\calO) \cap \calO| < \frac {|H|}{2t}$$
holds. Next, we used condition (3) of Theorem \ref{thmPrincipalk1}, to prove that if at least one of $\calO, \calO^c$ is not an independent subset, then 
$$|\theta_i(\calO) \cap \calO| \geq \frac {|H|}{2t}$$
holds, contradicting the above inequality. 
\end{remark}

\begin{proof}
[Proof of Theorem \ref{Thm:theta}]
By Remark \ref{Remark:Summary} and the preceding discussion, it suffices to prove that under the hypothesis of Theorem \ref{Thm:theta}, the inequality 
$$|\theta_i(\calO) \cap \calO| < \frac {|H|}{dt}$$
holds for any $1\leq i \leq d$, and in addition, the inequality 
$$|\theta_i(\calO) \cap \calO| \geq \frac {|H|}{dt}$$
holds if at least one of $\calO, \calO^c$ is not an independent subset. 

In the course of the proof of Theorem \ref{thmPrincipalk1}, using conditions (1), (2), (4) of Theorem \ref{thmPrincipalk1}, it is established that 
$$
|\theta_i(\calO)\cap \calO |
 \leq \left(\frac {2d\gamma}\varepsilon  + \sqrt{\frac{6d^2\gamma}{\varepsilon^2}} \right) \frac{|V|}{2}
 < \frac{|H|}{t}
$$
holds where 
$$\gamma =  d^2 \sqrt{2 
\left(\frac{\varepsilon^4}{2^9 d^8}\right)
\left(2- 
\left(\frac{\varepsilon^4}{2^9 d^8}\right)
\right)}.$$
As explained in Remark \ref{Remark}, under the hypothesis of Theorem \ref{Thm:theta}, conditions (1), (2), (4) of Theorem \ref{thmPrincipalk1} hold. Hence, the same argument implies that under the hypothesis of Theorem \ref{Thm:theta}, if we assume (which we do from now on) that the normalised adjacency operator of $(V, E)$ admits an eigenvalue in $[-1, -1 + \ell_{\varepsilon, d}]$ with 
$$
\ell_{\varepsilon, d}
= \frac{\varepsilon^4}{2^9 d^{10}},$$
then one obtains an index two subgroup $H$ of $\calG$ and using condition (1), it follows that the action of $H$ on $V$ has two orbits, and for an orbit $\calO$ among them, the inequality 
$$|\theta_i(\calO) \cap \calO| < \frac {|H|}{dt}$$
holds for any $1\leq i \leq d$. 

Note that for a vertex $v \in V$ and an automorphism $g$ of $(V, E)$, the image of a neighbour of $v$ under $g^\mo$, is a neighbour of $g^\mo (v)$, and hence the neighbour $g^\mo (\theta_j(v))$ of $g^\mo(v)$ is equal to $\theta_{i_{g, v, j}} (g^\mo(v))$ for some $1\leq i_{g, v, j} \leq d$. Consequently, for any automorphism $g$ of the graph $(V, E)$, and $v\in V$ and $1\leq j \leq d$, 
\begin{equation}
\label{Eqn:QuasiAuto}
\theta_{i_{g, v, j}} (g^\mo (v)) = g^\mo (\theta_j(v))
\end{equation}
holds for some $1\leq i_{g, v, j} \leq d$.

Suppose two vertices $u, v$ of $\calO^\dag$ are adjacent where $\calO^\dag$ is one of $\calO, \calO^c = V \setminus \calO$. Let $H^\dag$ denote $H$ (resp. $H^c = \calG \setminus H$) if $\calO^\dag  = \calO$ (resp. $\calO^\dag = \calO^c$). Assume that $u = \theta_j(v)$. Consider the map 
$$\psi_{v, j} : H^\dag \to \{1, 2, \ldots, d\},\quad
g \mapsto i_{g, v, j}.$$
Note that a fibre of this map having maximal size contains at least $\frac {|H^\dag|}d$ elements, and hence for some integer $1\leq i \leq d$, the inequality $|H^\ddag  |\geq \frac {|H^\dag|}d$ holds where $H^\ddag = \psi_{v, j}^\mo(i)$. Thus, for any $g\in H^\ddag$, it follows that $i_{g, v, j} = \psi_{v, j}(g) = i$, which yields 
\begin{equation}
\label{Eqn:QuasiAutoIJ}
\theta_i(g^\mo(v)) = g^\mo (\theta_j (v)) .
\end{equation}
This implies that 
\begin{align*}
|\theta_i(\calO)\cap \calO| 
& = 
|\theta_{i} (\calO)\cap \calO| \\
& = 
|\theta_{i} (H^\dag v) \cap H^\dag u|\\
& = |\{\theta_{i}(h^\mo v) \,|\, h\in H^\dag\}\cap H^\dag u|\\
& \geq |\{\theta_{i}(h^\mo v) \,|\, h\in H^\ddag\}\cap H^\dag u|\\
& \geq |\{h^\mo (\theta_j(v)) \,|\, h\in H^\ddag\}\cap H^\dag u|\\
& = |\{h^\mo (u) \,|\, h\in H^\ddag\}\cap H^\dag u|\\
& = |\{h^\mo (u) \,|\, h\in H^\ddag\}|\\
& \geq \frac{|H^\ddag|}{t} \\
& \geq \frac{|H^\dag|}{dt}\\
& = \frac{|H|}{dt}
\end{align*}
hold. Hence, $\calO$ and $\calO^c$ are independent subsets of $V$. So, the graph $(V, E)$ is bipartite, which contradicts the hypothesis. Consequently, the normalised adjacency operator of $(V, E)$ does not admit an eigenvalue in $[-1, -1 + \ell_{\varepsilon, d}]$
with 
$$
\ell_{\varepsilon, d}
= \frac{\varepsilon^4}{2^9 d^{10}}.$$
\end{proof}

Recall that a graph is said to be \textit{vertex-transitive} if its automorphism group acts transitively on its vertex set, i.e., for any two vertices $u, v$, there exists an automorphism $\varphi$ of the graph such that $\varphi(u) = v$. 

\begin{proof}[Proof of Theorem \ref{Thm:Intro}]
To establish Theorem \ref{Thm:Intro}, it suffices to prove that if the graph $(V, E)$ has degree $d$ and it is non-bipartite, vertex-transitive, and an $\varepsilon$-vertex expander for some $\varepsilon>0$, then the nontrivial eigenvalues of the normalised adjacency operator of $(V, E)$ are greater than $-1 + \frac{\varepsilon^4}{2^{9} d^{10}}$. 

By the Birkhoff-von Neumann theorem \cite[Theorem 5.5]{vanLintWilsonCourseCombi}, there exist permutations $\theta_1, \ldots, \theta_d: V\to V$ such that the vertices $v, \theta_i(v)$ are adjacent in $(V, E)$ for any $v\in V$ and $1\leq i \leq d$, and that $\calN(v)$ is equal to $\cup_{i=1}^d \{\theta_i(v)\}$ for any $v\in V$. Let $\calG$ be a subgroup of the automorphism group of the graph $(V, E)$ which acts transitively on $V$. Replacing $\calG$ by one of its subgroups, we may assume that no proper subgroup of $\calG$ acts transitively on $V$. It follows from Theorem \ref{Thm:theta} that the nontrivial eigenvalues of the normalised adjacency operator of $(V, E)$ are greater than $-1 + \frac{\varepsilon^4}{2^{9} d^{10}}$. 
\end{proof}

\subsection{Acknowledgements}
The work of the first author is supported by the ERC grant 716424 - CASe of K. Adiprasito. The second author acknowledges the INSPIRE Faculty Award (IFA18-MA123) from the Department of Science and Technology, Government of India. 

\def\cprime{$'$} \def\Dbar{\leavevmode\lower.6ex\hbox to 0pt{\hskip-.23ex
  \accent"16\hss}D} \def\cfac#1{\ifmmode\setbox7\hbox{$\accent"5E#1$}\else
  \setbox7\hbox{\accent"5E#1}\penalty 10000\relax\fi\raise 1\ht7
  \hbox{\lower1.15ex\hbox to 1\wd7{\hss\accent"13\hss}}\penalty 10000
  \hskip-1\wd7\penalty 10000\box7}
  \def\cftil#1{\ifmmode\setbox7\hbox{$\accent"5E#1$}\else
  \setbox7\hbox{\accent"5E#1}\penalty 10000\relax\fi\raise 1\ht7
  \hbox{\lower1.15ex\hbox to 1\wd7{\hss\accent"7E\hss}}\penalty 10000
  \hskip-1\wd7\penalty 10000\box7}
  \def\polhk#1{\setbox0=\hbox{#1}{\ooalign{\hidewidth
  \lower1.5ex\hbox{`}\hidewidth\crcr\unhbox0}}}
\providecommand{\bysame}{\leavevmode\hbox to3em{\hrulefill}\thinspace}
\providecommand{\MR}{\relax\ifhmode\unskip\space\fi MR }
\providecommand{\MRhref}[2]{%
  \href{http://www.ams.org/mathscinet-getitem?mr=#1}{#2}
}
\providecommand{\href}[2]{#2}

\end{document}